\documentclass[reqno]{amsart}

\usepackage{hyperref}

\begin{document}

\title[A Functional Integro-differential Fractional Equation]{Existence
and Uniqueness of Solution\\ to a Functional Integro-differential\\ Fractional Equation}


\author[M. R. Sidi Ammi, E. H. El Kinani, D. F. M. Torres]{Moulay Rchid Sidi Ammi,
El Hassan El Kinani, Delfim F. M. Torres}


\address{Moulay Rchid Sidi Ammi \newline
Faculty of Sciences and Techniques,
Moulay Ismail University, AMNEA Group, \newline
Errachidia, Morocco}
\email{sidiammi@ua.pt}


\address{El Hassan El Kinani \newline
Faculty of Sciences and Techniques,
Moulay Ismail University, GAA Group, \newline
Errachidia, Morocco}
\email{\url{elkinani_67@yahoo.com}}


\address{Delfim F. M. Torres \newline
Center for Research and Development in Mathematics and Applications (CIDMA), \newline
Department of Mathematics,
University of Aveiro,
Campus Universit\'{a}rio de Santiago,
3810-193 Aveiro, Portugal}
\email{delfim@ua.pt}


\thanks{Submitted 27-Jan-2012; accepted  18-Jun-2012; 
for publication in the Electronic Journal of Differential Equations
(\url{http://ejde.math.txstate.edu}).}




\begin{abstract}
We prove, using a fixed point theorem in a Banach algebra,
an existence result for a fractional functional differential
equation in the Riemann--Liouville sense. Dependence of
solutions with respect to initial data and an uniqueness result
are also derived.
\end{abstract}


\maketitle

\numberwithin{equation}{section}

\newtheorem{theorem}{Theorem}[section]
\newtheorem{lemma}[theorem]{Lemma}
\newtheorem{definition}[theorem]{Definition}
\newtheorem{corollary}[theorem]{Corollary}

\allowdisplaybreaks


\section{Introduction}

Fractional Calculus is a generalization of ordinary
differentiation and integration to arbitrary (non-integer) order.
The subject has its origin in the 1600s. During three centuries,
the theory of fractional calculus developed as a pure theoretical
field, useful only for mathematicians. In the last few decades,
however, fractional differentiation proved very useful in various
fields of applied sciences and engineering: physics (classic and
quantum mechanics), chemistry, biology, economics, signal and
image processing, calculus of variations, control theory,
electrochemistry, visco-elasticity, feedback amplifiers, and
electrical circuits
\cite{MR2558546,gkk,hil,kst,book:Ortigueira,sat,skm,srs,rhfc}. The
``bible'' of fractional calculus is the book of Samko, Kilbas and
Marichev \cite{skm}.

Several definitions of fractional derivatives are available in the
literature, including the Riemann--Liouville, Grunwald--Letnikov,
Caputo, Riesz, Riesz--Caputo, Weyl, Hadamard, and Chen derivatives
\cite{kli,mst,mir,pod1,rsl,skm}. The most common used fractional
derivative is the Riemann--Liouville \cite{MR2821461,rsl,skm,Russ},
which we adopt here. It is worth to mention that functions that have no
first order derivative might have Riemann--Liouville fractional
derivatives of all orders less than one \cite{rsl}. Recently, the
physical meaning of the initial conditions to fractional
differential equations with Riemann--Liouville derivatives has
been discussed \cite{gia,MyID:181,pod1}.

Using a fixed point theorem, like Schauder's fixed point theorem,
and the Banach contraction mapping principle, several results of
existence have been obtained in the literature to linear and
nonlinear equations, and recently also to fractional differential
equations. The interested reader is referred to
\cite{sayed,agarwal1,agarwal,st}.

Let $I_{0}=[-\delta, 0]$ and $I=[0,T]$ be two closed and bounded
intervals in $\mathbb{R}$; $B(I, \mathbb{R})$ be the space of
bounded real-valued functions on $I$; and $C=C(I_0, \mathbb{R})$
be the space of continuous real valued functions on $I_{0}$. Given
a function $\phi \in C$, we consider the functional
integro-differential fractional equation
\begin{equation}
\label{eq1}
\frac{d^{\alpha}}{dt^{\alpha}} \left[ \frac{x(t)}{f(t,
x(t))} \right] = g \left( t, x_{t}, \int_{0}^{t}k(s, x_{s}) ds
\right) \, a.e., \quad t \in I,
\end{equation}
subject to
\begin{equation}
\label{equa1}
x(t)= \phi (t), \quad t \in I_{0},
\end{equation}
where $d^{\alpha}/dt^{\alpha}$ denotes the Riemann--Liouville
derivative of order $\alpha$, $0 < \alpha < 1$,
and $x_{t}: I_{0}\rightarrow \mathbb{C}$ is the continuous function
defined by $x_{t}(\theta)=x(t+\theta)$ for all $\theta \in I_{0}$,
under suitable mixed Lipschitz and other conditions on the
nonlinearities $f$ and $g$. For a motivation to study such type of
problems we refer to \cite{sayed}. Here we just mention that
problems of type \eqref{eq1}--\eqref{equa1} seem important in the
study of dynamics of biological systems \cite{sayed}.

Our main aim is to prove existence of solutions for
\eqref{eq1}--\eqref{equa1}. This is done in
Section~\ref{sec:exist} (Theorem~\ref{thm2}). Our main tool is a
fixed point theorem that is often useful in proving existence
results for integral equations of mixed type in Banach algebras
\cite{dhage}, and which we recall in Section~\ref{sec:prelim}
(Theorem~\ref{thm1}). We end with Section~\ref{sec4}, by proving
dependence of the solutions with respect to their initial
values (Theorem~\ref{thm:unif:stab}) and, consequently,
uniqueness to \eqref{eq1}--\eqref{equa1} (Corollary~\ref{cor:unique}).


\section{Preliminaries}
\label{sec:prelim}

In this section we give the notations, definitions, hypotheses and
preliminary tools, which will be used in the sequel. We deal with
the (left) Riemann--Liouville fractional derivative, which is
defined in the following way.

\begin{definition}[\cite{skm}]
\label{def1} The fractional integral of order $\alpha \in (0,1)$
of a function $f \in L^1[0,T]$ is defined by
\[
I^{\alpha} f(t) = \int_0^t \frac{(t - s)^{ \alpha-1}}{\Gamma(
\alpha)} f(s)\, ds,
\]
$t \in [0,T]$, where $\Gamma$ is the Euler gamma function. The
Riemann--Liouville fractional derivative operator of order
$\alpha$ is then defined by
$$
\frac{d^{\alpha}}{dt^{\alpha}} :=\frac{d}{dt} \circ I^{1-\alpha} .
$$
\end{definition}

Along the paper, $X$ denotes a Banach algebra with norm $\| \cdot
\|$. The space $C(I, \mathbb{R})$ of all continuous functions
endowed with the norm $\|x\|= \sup_{t \in I}|x(t)|$ is a Banach
algebra. To prove the existence result for
\eqref{eq1}--\eqref{equa1}, we shall use the following fixed point
theorem.

\begin{theorem}[\cite{dhage}]
\label{thm1} Let $B_{r}(0)$ and $\overline{B_{r}(0)}$ be,
respectively, open and closed balls in a Banach algebra $X$
centered at origin 0 and of radius $r$. Let $A, B:
\overline{B_{r}(0)} \rightarrow X$ be two operators satisfying:
 \begin{itemize}

\item[(a)] $A$ is Lipschitz with Lipschitz constant $L_{A}$,

\item[(b)] B is compact and continuous, and

\item[(c)] $L_{A} M < 1$, where
$M= \|B(\overline{B_{r}(0)})\| := \sup \{ \|Bx\|; x \in
\overline{B_{r}(0)} \}$.
\end{itemize}
Then, either
\begin{itemize}

\item[(i)] the equation $\lambda [AxBx]=x$ has a solution for
$\lambda =1$, or

\item[(ii)] there exists  $x \in X$ such that $\|x\|=r$,
$\lambda [AxBx]=x$ for some $0< \lambda <1$.
\end{itemize}
\end{theorem}

Throughout the paper, we assume the following hypotheses:
\begin{itemize}

\item[(H1)] Function $f: I \times \mathbb{R} \rightarrow
\mathbb{R} \setminus \{0\}$ is Lipschitz and there exists a
positive constant $L$ such that for all $x, y \in \mathbb{R}$,
$|f(t, x)-f(t, y)| \leq L |x-y|$ a.e., $t \in I$.
\item[(H2)] Function $k: I \times C \rightarrow
\mathbb{R}$ is continuous and there exists a function $\beta \in
L^{1}(I, \mathbb{R}^{+})$ such that $|k(s,y)| \leq \beta (s)
\|y\|$ a.e., $s \in I, y \in C$.

\item[(H3)] There exists a continuous function $\gamma \in
L^{\infty}(I, \mathbb{R}^{+})$ and a contraction  $\psi :
\mathbb{R}^{+}\rightarrow \mathbb{R}^{+}$ with contraction
constant $<1$ and $\psi(0)=0$ such that for $x \in C$ and $y \in
\mathbb{R}$, $g(t, x, y) \leq \gamma(t) \psi(\|x\|+ |y|)$  a.e.,
$t \in I$.

\item[(H4)] Function $k$ is Lipschitz with respect to the second variable
with Lipschitz constant $L_{k}$: $|k(s, x)-k(s, y)| \leq
L_{k}\|x-y\|$ for $s\in I, x, y \in C$.

\item[(H5)] There exist $L_{1}$ and $L_{2}$ such
that for $x_{1}$, $x_{2} \in C$, $y_{1}$, $y_{2} \in \mathbb{R}$
and $s\in I$,
$$
\left|g(s, x_{1}, y_{1})-g(s, x_{2}, y_{2})\right| \leq L_{1}
\left\|x_{1}-x_{2}\right\| + L_{2} \left|y_{1}-y_{2}\right|.
$$
\end{itemize}


\section{Existence of solution}
\label{sec:exist}

We prove existence of a solution to \eqref{eq1}--\eqref{equa1}
under hypotheses (H1)--(H5).

\begin{theorem}
\label{thm2}
Suppose that hypotheses (H1)--(H5) hold. Assume
there exists a real number $r>0$ such that
\begin{equation}
\label{eq3} r> \frac{\frac{F T^{\alpha}}{\Gamma( \alpha +1
)}\sup_{s \in (0, T)} \gamma (s) (1+\|\beta\|_{L^{1}} ) \psi (
r)}{ 1- \frac{LT^{\alpha}}{\Gamma( \alpha +1 )} \sup_{s \in (0,
T)} \gamma(s) (1+\|\beta\|_{L^{1}} ) \psi ( r)},
\end{equation}
where
$$
1-\frac{LT^{\alpha}}{\Gamma( \alpha +1 )} \sup_{s \in (0, T)}
\gamma(s) (1+\|\beta\|_{L^{1}} ) \psi (r) > 0, \quad F=\sup_{t \in
[0, T]}|f(t, 0)|.
$$
Then, problem \eqref{eq1}--\eqref{equa1} has a solution on $I$.
\end{theorem}

\begin{proof}
Let $X= C(I, \mathbb{R})$. Define an open ball $B_{r}(0)$ centered
at origin and of radius $r>0$, which satisfies \eqref{eq3}. It
is easy to see that $x$ is a solution to \eqref{eq1}--\eqref{equa1}
if and only if it is a solution of the integral equation
\begin{equation*}
x(t)= f(t, x(t)) I^{\alpha}\left( g \left( t, x_{t}, \int_{0}^{t}
k(s, x_{s}) ds \right) \right).
\end{equation*}
In other terms,
\begin{equation}
\label{eq4} x(t) = f(t, x(t)) \int_0^t \frac{(t - s)^{
\alpha-1}}{\Gamma( \alpha)} g\left(s, x_{s}, \int_{0}^{s} k(\tau,
x_{\tau}) d\tau \right)ds.
\end{equation}
Integral equation \eqref{eq4} is equivalent to the operator
equation $Ax(t) Bx(t)= x(t)$, $t \in I$, where $A, B:
\overline{B_{r}(0)} \rightarrow X$ are defined by
$$
Ax(t)= f(t, x(t)) \ \mbox{ and }  \ Bx(t)= I^{\alpha}\left(g
\left( t, x_{t}, \int_{0}^{t} k(s, x_{s}) ds \right) \right).
$$
We need to prove that the operators $A$ and $B$ verify the
hypotheses of Theorem~\ref{thm1}. To continue the proof of
Theorem~\ref{eq3}, we make use of two technical lemmas.

\begin{lemma}
\label{lem31} The operator $A$ is Lipschitz on $X$.
\end{lemma}

\begin{proof}
Let $x, y \in X$ and $t \in I$. By (H1) we have
\begin{equation*}
|Ax(t)-Ay(t)| = |f(t, x)-f(t, y)| \leq L |x(t)-y(t)| \leq L
\|x-y\|.
\end{equation*}
Then, $\|Ax-Ay \| \leq  L \|x-y\|$ and it follows that $A$ is
Lipschitz on $X$ with Lipschitz constant $L$.
\end{proof}

\begin{lemma}
\label{lem33} The operator $B$ is completely continuous on $X$.
\end{lemma}

\begin{proof}
We prove that $B(\overline{B_{r}(0)})$ is an uniformly bounded and
equicontinuous set in $X$. Let $x$ be arbitrary in
$\overline{B_{r}(0)}$. By hypotheses (H2)--(H4) we have
\begin{align*}
|Bx(t)| & \leq  I^{\alpha}\left(
\left|g \left( t, x_{t}, \int_{0}^{t}k(s, x_{s}) ds \right)\right| \right)\\
& \leq  \frac{1}{\Gamma( \alpha)} \int_{0}^{t} (t - s)^{ \alpha-1}
\left|g \left( s, x_{s}, \int_{0}^{s}k(\tau, x_{\tau})d \tau \right) \right| ds\\
& \leq  \frac{1}{\Gamma( \alpha)} \int_{0}^{t} (t - s)^{ \alpha-1}
\gamma(s) \psi\left(\|x_{s}\|+ \int_{0}^{s} \beta (\tau)\|x_{\tau}\| d \tau \right) ds\\
& \leq  \frac{1}{\Gamma( \alpha)} \int_{0}^{t}
(t - s)^{ \alpha-1} \gamma (s) (1+\|\beta\|_{L^{1}}) \psi(r) ds\\
& \leq \frac{\sup_{s\in [0, T]} \gamma(s)}{\Gamma( \alpha)}
\left(1+\|\beta\|_{L^{1}}\right) \psi(r) \int_{0}^{t} (t - s)^{ \alpha-1} ds\\
& \leq \frac{\sup_{s\in [0, T]} \gamma(s)}{\Gamma( \alpha +1)}
\left(1+\|\beta\|_{L^{1}}\right) \psi(r) T^{\alpha}.
\end{align*}
Taking the supremum over $t$, we get $\|Bx\| \leq M$ for all $x
\in \overline{B_{r}(0)}$, where
$$
M= \frac{\sup_{s\in [0, T]} \gamma(s)}{\Gamma( \alpha +1)}
(1+\|\beta\|_{L^{1}}) \psi(r) T^{\alpha} .
$$
It results that $B(\overline{B_{r}(0)})$ is an uniformly bounded
set in $X$. Now, we shall prove that $B(\overline{B_{r}(0)})$ is
an equicontinuous set in $X$. For $0 \leq t_{1} \leq t_{2} \leq T$
we have
\begin{align*}
|Bx(t_{2}) &- Bx(t_{1})| \\
& \leq \frac{1}{\Gamma( \alpha)} \Bigg\{ \Bigg| \int_{0}^{t_{2}}
(t_{2} - s)^{ \alpha-1}
g\left(s, x_{s}, \int_{0}^{s} k(\tau, x_{\tau}) d\tau \right) ds\\
& \qquad \qquad \qquad - \int_{0}^{t_{1}} (t_{1} - s)^{ \alpha-1}
g\left( s, x_{s}, \int_{0}^{s} k(\tau, x_{\tau}) d\tau \right) ds  \Bigg| \Bigg\} \\
& \leq   \frac{1}{\Gamma( \alpha)} \Bigg\{  \Bigg|
\int_{0}^{t_{1}} (t_{2} - s)^{ \alpha-1}
g\left(s, x_{s}, \int_{0}^{s} k(\tau, x_{\tau}) d\tau \right) ds \\
& \qquad \qquad \qquad + \int_{t1}^{t_{2}} (t_{2} - s)^{ \alpha-1}
g\left(s, x_{s}, \int_{0}^{s} k(\tau, x_{\tau}) d\tau \right) ds\\
& \qquad \qquad \qquad - \int_{0}^{t_{1}} (t_{1} - s)^{ \alpha-1}
g\left(s, x_{s}, \int_{0}^{s} k(\tau, x_{\tau}) d\tau \right) ds  \Bigg| \Bigg\} \\
& \leq \frac{1}{\Gamma( \alpha)} \int_{0}^{t_{1}}\left| (t_{2} -
s)^{ \alpha-1}- (t_{1} - s)^{ \alpha-1}\right|
\, \left|g \left( s, x_{s}, \int_{0}^{s} k(\tau, x_{\tau}) d\tau \right)\right| ds \\
& \qquad \qquad \qquad + \frac{1}{\Gamma( \alpha)}
\int_{t_{1}}^{t_{2}} \left| (t_{2} - s)^{ \alpha-1}\right| \left|g
\left( s, x_{s}, \int_{0}^{s} k(\tau, x_{\tau}) d\tau
\right)\right| ds.
\end{align*}
On the other hand,
\begin{align*}
\left|g\left(s, x_{s}, \int_{0}^{s} k(\tau, x_{\tau}) d\tau
\right) ds \right| & \leq \gamma(s) \psi \left( \|x\|
+ \left|\int_{0}^{s} k(\tau, x_{\tau} ) d \tau \right| \right)\\
& \leq \gamma(s) \psi\left( \|x\|+ \int_{0}^{s} \beta (\tau ) \|x\| d \tau \right) \\
& \leq \gamma(s)  \|x\| (1+  \|\beta \|_{L^{1}} ) \psi (s)\\
& \leq \sup_{s}(\gamma(s)\psi (s))  \|x\| \left(1+  \|\beta
\|_{L^{1}} \right) \leq c,
\end{align*}
where $c$ is a positive constant. Then,
\begin{align*}
|Bx(t_{2}) &-Bx(t_{1})|\\
& \leq  \frac{c}{\Gamma( \alpha)} \int_{0}^{t_{1}}\left| (t_{2} -
s)^{ \alpha-1}- (t_{1} - s)^{ \alpha-1} \right| ds +
\frac{c}{\Gamma( \alpha)} \int_{t_{1}}^{t_{2}}
\left| (t_{2} - s)^{ \alpha-1}\right| ds \\
& \leq \frac{c}{\Gamma( \alpha +1)} \left|t_{2}^{\alpha} -
t_{1}^{\alpha}-2 (t_{2}- t_{1})^{\alpha}\right|.
\end{align*}
Because the right hand side of the above inequality doesn't depend
on $x$ and tends to zero when $t_{2}\rightarrow t_{1}$, we
conclude that $B(\overline{B_{r}(0)})$ is relatively compact.
Hence, $B$ is compact by the Arzela--Ascoli theorem. It remains to
prove that $B$ is continuous. For that, let us consider a sequence
$x^{n}$ converging to $x$. Then,
\begin{align*}
&|F x^{n}(t) - Fx(t)|\\
& \leq  \Biggl| f(t, x^{n})  I^{\alpha}\left( g\left( t,
x^{n}_{t}, \int_{0}^{t}k(s, x^{n}_{s})ds \right)\right)
- f(t, x) I^{\alpha}\left( g \left( t, x_{t},
\int_{0}^{t}k(s, x_{s})ds \right)\right)\Biggr|\\
& \leq  \left|f(t, x^{n})-f(t, x)\right| I^{\alpha}\left(
\left|g \left( t, x^{n}_{t}, \int_{0}^{t}k(s, x^{n}_{s}) ds\right)\right|  \right)\\
& \qquad \quad + |f(t, x)| \Biggl|  I^{\alpha}\left( g \left( t,
x^{n}_{t}, \int_{0}^{t}k(s, x^{n}_{s}) ds\right) \right) -
I^{\alpha}\left( g \left(t,
x_{t}, \int_{0}^{t}k(s, x_{s}) ds\right) \right) \Biggr|\\
& \leq  L \|x^{n}-x\| + |f(t, x)| I^{\alpha}\Biggl( \Biggl|g
\left( t,x^{n}_{t}, \int_{0}^{t}k(s, x^{n}_{s}) ds\right) -g
\left( t,x_{t}, \int_{0}^{t}k(s, x_{s}) ds \right) \Biggr|
\Biggr).
\end{align*}
On the other hand,
\begin{align*}
I^{\alpha}&\left( \left| g\left( t, x^{n}_{t},
\int_{0}^{t}k(s,x^{n}_{s}) ds \right)
-g \left( t, x_{t}, \int_{0}^{t}k(s,x_{s}) ds \right) \right| \right) \\
& \leq  I^{\alpha}\left( L_{1}|x^{n}_{t}-x_{t}| + L_{2}
\left|\int_{0}^{t}k(s, x^{n}_{s}) ds - \int_{0}^{t}k(s,x_{s}) ds \right| \right) \\
& \leq  I^{\alpha}\left( L_{1}|x^{n}_{t}-x_{t}|
+ L_{2} \int_{0}^{t} \left| k(s, x^{n}_{s})- k(s,x_{s})\right|ds \right) \\
& \leq  I^{\alpha}\left( L_{1}\left|x^{n}_{t}-x_{t}\right|
+ L_{2}L_{k} \int_{0}^{t} \left| x^{n}_{s}-x_{s}\right|ds \right) \\
& \leq  I^{\alpha}\left( L_{1}\|x^{n}_{t}-x_{t}\|
+ L_{2}L_{k}T \| x^{n}_{s}-x_{s}\| \right) \\
& \leq  \left( L_{1}+ L_{2}L_{k}T\right)\|x^{n}_{t}-x_{t}\|
\int_{0}^{t} \frac{(t-s)^{\alpha -1}}{\Gamma(\alpha)} ds\\
& \leq  \frac{1}{\Gamma(\alpha +1)}\left( L_{1}+
L_{2}L_{k}T\right)\|x^{n}_{t}-x_{t}\|.
\end{align*}
Taking the norm, $\|Fx^{n}- Fx\|  \leq  L \|x^{n}-x\| +
\frac{1}{\Gamma(\alpha +1)} |f(t, x)| \left( L_{1}+
L_{2}L_{k}T\right)\|x^{n}_{t}-x_{t}\|$. Hence, the right hand side
of the above inequality tends to zero whenever $x^{n}\rightarrow
x$. Therefore, $Fx^{n} \rightarrow Fx$. This proves the continuity
of $F$.
\end{proof}

Using Theorem~\ref{thm1}, we obtain that either the conclusion (i)
or (ii) holds. We show that item (ii) of Theorem~\ref{thm1} cannot
be realizable. Let $x \in X$ be such that $\|x\|=r$ and $x(t)=
\lambda f(t, x(t)) I^{\alpha}\left( g\left( t,
x_{t},\int_{0}^{t}k(s, x_{s}) ds \right) \right)$ for any $\lambda
\in (0, 1)$ and $t \in I$. It follows that
\begin{align*}
|x(t)| & \leq   \lambda \left(  |f(t, x(t))- f(t, 0)|+|f(t,
0)|\right)
I^{\alpha}\left( g\left( t, x_{t}, \int_{0}^{t}k(s, x_{s})  ds \right)\right) \\
& \leq \lambda  \left( L |x(t)|+F \right) I^{\alpha}\left(
g\left( t, x_{t}, \int_{0}^{t}k(s, x_{s}) ds \right)\right) \\
& \leq  \lambda \left( L \|x\|+F \right)
I^{\alpha}\left( g \left( t, x_{t}, \int_{0}^{t}k(s, x_{s}) ds \right) \right) \\
& \leq  \left(L \|x\| + F\right) I^{\alpha}\left(
\gamma (t) \psi \left( \|x\| + \left|\int_{0}^{t} k(s, x_{s}) ds \right| \right) \right)\\
& \leq \left(L \|x\| + F\right) I^{\alpha}\left(
\gamma (t) \psi\left( \|x\| +  \|\beta\|_{L^{1}} \|x\| \right) \right)\\
& \leq  \frac{L \|x\| + F}{\Gamma( \alpha)} \int_{0}^{t}
\gamma(s) \left(1+\|\beta\|_{L^{1}}\right)
\psi\left(\|x\|\right)(t-s)^{ \alpha-1} ds \\
& \leq  \frac{L \|x\| + F}{\Gamma( \alpha)} \sup_{s \in (0, T)}
\gamma(s) \left(1+\|\beta\|_{L^{1}}\right) \psi\left( \|x\|\right)
\int_{0}^{t} (t - s)^{ \alpha-1} ds \\
& \leq  \frac{L \|x\| + F}{\Gamma( \alpha +1 )} \sup_{s \in (0,
T)}
\gamma (s) \left(1+\|\beta\|_{L^{1}}\right) \psi\left( \|x\|\right) T^{\alpha}\\
& \leq  \left( \frac{LT^{\alpha}}{\Gamma( \alpha +1 )} \sup_{s \in
(0, T)}  \gamma (s) \left(1+\|\beta\|_{L^{1}}\right)
\psi\left(\|x\|\right) \right)\|x\|\\
& \qquad \quad + \frac{F T^{\alpha}}{\Gamma( \alpha +1 )}\sup_{s
\in (0, T)} \gamma(s) \left(1+\|\beta\|_{L^{1}}\right) \psi\left(
\|x\|\right).
\end{align*}
Passing to the supremum in the above inequality, we obtain
\begin{equation}
\label{e} \|x\| \leq \frac{\frac{F T^{\alpha}}{\Gamma( \alpha +1
)}\sup_{s \in (0, T)}  \gamma (s) \left(1+\|\beta\|_{L^{1}}\right)
\psi( \|x\|)}{1-\frac{LT^{\alpha}}{\Gamma( \alpha +1 )} \sup_{s
\in (0, T)} \gamma(s) \left(1+\|\beta\|_{L^{1}}\right) \psi (
\|x\|)}.
\end{equation}
If we replace $\|x\| = r$ in \eqref{e}, we have
$$
r \leq \frac{\frac{F T^{\alpha}}{\Gamma( \alpha +1 )}\sup_{s \in
(0, T)}  \gamma (s) (1+\|\beta\|_{L^{1}} ) \psi ( r)}{ 1-
\frac{LT^{\alpha}}{\Gamma( \alpha +1 )} \sup_{s \in (0, T)} \gamma
(s) (1+\|\beta\|_{L^{1}} ) \psi(r)},
$$
which is in contradiction to \eqref{eq3}. Then the conclusion (ii)
of Theorem~\ref{thm1} is not possible. Therefore, the operator
equation $AxBx=x$ and, consequently, problem
\eqref{eq1}--\eqref{equa1}, has a solution on $I$. This ends the
proof of Theorem~\ref{thm2}.
\end{proof}

Let us see an example of application of our Theorem~\ref{thm2}.
Let $I_{0}=[-\pi, 0]$ and $I=[0, \pi]$. Consider
the integro-differential fractional equation
\begin{equation}
\label{eq6} \frac{d^{\alpha}}{dt^{\alpha}} \left(
\frac{x(t)}{1+\frac{\sin t}{12} |x(t)|} \right) =  g\left( t,
x_{t}\left(\frac{t}{2}\right), \int_{0}^{t}k(s, x_{s})ds \right) \
a.e., \quad t \in I,
\end{equation}
where $\alpha =\frac{1}{2}$ and $f: I\times \mathbb{R} \rightarrow
\mathbb{R}^{+} \setminus \{0\}$, $g: I \times C \times
\mathbb{R}\rightarrow \mathbb{R}$ and $k: I \times C\rightarrow
\mathbb{R}$ are given by
\[
f(t, x)= 1+\frac{\sin t}{12}(|x(t)|)\, , \quad g(t, x, y) =
\gamma(t) ( \|x\| + |y|) \, , \quad k(t, x)= \frac{x}{4 \pi } \, ,
\]
with $\gamma(t)=t, \beta(t)=\frac{1}{4\pi}, T=\pi, L=\frac{1}{12},
F=\sup_{t}f(t, 0)=1, \|\beta\|_{L^{1}}=\frac{1}{4}$ and $\sup_{t
\in [0, \pi] }\gamma(t)=\pi$. It is easy to see that all
hypotheses (H1)--(H5) are  satisfied with $\psi(r)=\frac{r}{12}$
for all $r \in \mathbb{R}^{+}$. By Theorem~\ref{thm2}, $r$ satisfies
$$
\frac{12-FB}{LB}\simeq 0,26 \leq r \leq \frac{12}{LB}\simeq 18,34,
$$
where $B=\frac{T^{\alpha}}{\Gamma(\alpha+1)}\sup_{s \in (0, T)}\gamma (s)
(1+\|\beta\|_{L^{1}} )$. We conclude that if $r=2$, then
\eqref{eq6} has a solution in $\overline{B_{2}(0)}$.


\section{Dependence on the data and uniqueness of solution}
\label{sec4}

In this section we derive uniqueness
of solution to \eqref{eq1}--\eqref{equa1}.

\begin{theorem}
\label{thm:unif:stab}
Let $x$ and $y$ be two solutions to the
nonlocal fractional equation \eqref{eq1} subject to
\eqref{equa1} with $\phi = \phi_{1}$ and
$\phi = \phi_{2}$, respectively. Then, we have
\begin{align*}
\|x - y \| & \leq \frac{(L \|y\| +F)(L_{1} +L_{2}L_{k} T )
\frac{T^{\alpha}}{\Gamma( \alpha +1 )}}{1-L} \| \phi_{1} -\phi_{2}
\|.
\end{align*}
\end{theorem}

\begin{proof}
Let $x$ and $y$ be two solutions of \eqref{eq1}.
Then, from \eqref{eq4}, one has
\begin{align*}
|x&(t) - y(t)|  \\
&\leq  \left| f(t, x)  I^{\alpha}\left( g\left(t, x_{t},
\int_{0}^{t}k(s, x_{s})ds \right)\right) - f(t, y)
I^{\alpha}\left(g\left(t, y_{t},
\int_{0}^{t}k(s, y_{s})ds\right)\right)\right|\\
& \leq  |f(t, x)-f(t, y)| I^{\alpha}\left(
\left|g\left(t, x_{t}, \int_{0}^{t}k(s, x_{s})ds\right)\right| \right)\\
& \qquad + |f(t, y)| \left|  I^{\alpha}\left( g\left(t, x_{t},
\int_{0}^{t}k(s, x_{s})ds\right)\right)
- I^{\alpha}\left( g\left( t, y_{t}, \int_{0}^{t}k(s, y_{s})ds \right) \right) \right|\\
& \leq  L \|x-y\| + |f(t, y)| I^{\alpha}\left( \left|g \left(t,
x_{t}, \int_{0}^{t}k(s, x_{s})ds\right)-g\left(t, y_{t},
\int_{0}^{t}k(s, y_{s}) ds \right) \right|\right)
\end{align*}
for $t>0$. On the other hand, we have
\begin{align*}
\Biggl|g\biggl(t, x_{t}, &\int_{0}^{t}k(s, x_{s}) ds\biggr)
-g\left(t, y_{t}, \int_{0}^{t}k(s, y_{s}) ds\right) \Biggr| \\
& \leq  L_{1}|x_{t}-y_{t}| + L_{2}\int_{0}^{t} |k(s,x_{s})-k(s,y_{s})| ds \\
& \leq  L_{1} |\phi_{1}(t) -\phi_{2}(t)| + L_{2}L_{k} T \| \phi_{1} -\phi_{2} \|\\
& \leq \left(  L_{1} +L_{2} L_{k}T \right) \| \phi_{1} -\phi_{2}
\|.
\end{align*}
Then,
\begin{align*}
|f(t, y)| & I^{\alpha}\left(\left| g \left( t, x_{t},
\int_{0}^{t}k(s, x_{s}) ds \right)
-g \left( t, y_{t}, \int_{0}^{t}k(s, y_{s}) ds \right)\right|\right)\\
& \leq |f(t, y)| \left(  L_{1} +L_{2}L_{k} T \right) \| \phi_{1}
-\phi_{2} \|
\int_{0}^{t} \frac{(t-s)^{\alpha -1}}{\Gamma( \alpha )} ds\\
& \leq |f(t, y)| \left(  L_{1} +L_{2} L_{k} T \right) \left\|
\phi_{1} -\phi_{2}\right\| \frac{T^{\alpha}}{\Gamma( \alpha +1 )}.
\end{align*}
Taking the supremum, we conclude that
\begin{align*}
\|x &- y \|  \leq L \|x-y\| + |f(t, y)| \left(  L_{1} +L_{2}L_{k}
T \right)
\| \phi_{1} -\phi_{2} \| \frac{T^{\alpha}}{\Gamma( \alpha )}\\
& \leq  L \|x-y\| + \left(|f(t, y)-f(t,0)|+|f(t,0)| \right) \left(
L_{1} +L_{2}L_{k} T \right)
\| \phi_{1} -\phi_{2} \| \frac{T^{\alpha}}{\Gamma( \alpha +1)}\\
& \leq  L \|x-y\| + \left(L \|y\| +\sup_{t} |f(t, 0)|\right)
\left(L_{1} +L_{2}L_{k} T\right)\frac{T^{\alpha}}{\Gamma( \alpha
+1 )}
\| \phi_{1} -\phi_{2} \| \\
& \leq  L \|x-y\| +\left(L \|y\| +F\right)\left(L_{1} +L_{2}L_{k}
T\right) \frac{T^{\alpha}}{\Gamma( \alpha +1 )}\| \phi_{1}
-\phi_{2} \|.
\end{align*}
Therefore,
\begin{align*}
\|x - y \| & \leq \frac{(L \|y\| +F)(L_{1} +L_{2}L_{k} T )
\frac{T^{\alpha}}{\Gamma( \alpha +1 )}}{1-L} \| \phi_{1} -\phi_{2}
\|.
\end{align*}
\end{proof}

\begin{corollary}
\label{cor:unique}
The solution predicted by Theorem~\ref{thm2} is unique.
\end{corollary}


\subsection*{Acknowledgments}

This work was supported by {\it FEDER} through {\it COMPETE}
--- Operational Programme Factors of Competitiveness (``Programa
Operacional Factores de Competitividade'') and by Portuguese funds
through the {\it Center for Research and Development in
Mathematics and Applications} (University of Aveiro) and the
Portuguese Foundation for Science and Technology (``FCT ---
Funda\c{c}\~{a}o para a Ci\^{e}ncia e a Tecnologia''), within
project PEst-C/MAT/UI4106/2011 with COMPETE number
FCOMP-01-0124-FEDER-022690. The authors were also supported by the
project \emph{New Explorations in Control Theory Through Advanced
Research} (NECTAR) cofinanced by FCT, Portugal, and the
\emph{Centre National de la Recherche Scientifique et Technique}
(CNRST), Morocco.



\end{document}